\documentclass{amsart}
\usepackage{paralist}
\usepackage{mathtools}
\usepackage{amsmath,amssymb,amsthm,color,mathrsfs,array}
\usepackage{enumitem}
\usepackage{multirow}
\usepackage{verbatim}
\usepackage{tabularx}
\usepackage{url}

\def\nd{{\mathbb{N}_d^n}}
\def\ot{{\otimes d}}
\newcommand{\re}{\mathbb{R}}
\newcommand{\mR}{\mathbb{R}}
\newcommand{\mA}{\mathcal{A}}
\newcommand{\mB}{\mathcal{B}}

\newcommand{\mr}{\mathbb{R}}

\newcommand{\N}{\mathbb{N}}

\newcommand{\A}{\mathcal{A}}

\newcommand{\lmd}{\lambda}

\newcommand{\eps}{\epsilon}

\newcommand{\dt}{\delta}

\def\af{\alpha}

\def\rank{\mbox{rank}}

\newcommand{\sig}{\sigma}
\newcommand{\Sig}{\Sigma}

\newcommand{\reff}[1]{(\ref{#1})}

\newcommand{\mc}[1]{\mathcal{#1}}
\newcommand{\mt}[1]{\mathtt{#1}}

\newcommand{\supp}[1]{\mbox{supp}(#1)}

\newcommand{\qmod}[1]{\mbox{QM}[#1]}
\newcommand{\ideal}[1]{\mbox{Ideal}[#1]}
\newcommand{\st}{\mathit{s.t.}}

\newcommand{\f}{\mathcal{F}}

\newcommand{\bdes}{\begin{description}}
	\newcommand{\edes}{\end{description}}
\newcommand{\bal}{\begin{align}}
	\newcommand{\eal}{\end{align}}
\newcommand{\bnum}{\begin{enumerate}}
	\newcommand{\enum}{\end{enumerate}}
\newcommand{\bit}{\begin{itemize}}
	\newcommand{\eit}{\end{itemize}}
\newcommand{\bea}{\begin{eqnarray}}
	\newcommand{\eea}{\end{eqnarray}}
\newcommand{\be}{\begin{equation}}
	\newcommand{\ee}{\end{equation}}
\newcommand{\baray}{\begin{array}}
	\newcommand{\earay}{\end{array}}
\newcommand{\bsry}{\begin{subarray}}
	\newcommand{\esry}{\end{subarray}}
\newcommand{\bca}{\begin{cases}}
	\newcommand{\eca}{\end{cases}}
\newcommand{\bcen}{\begin{center}}
	\newcommand{\ecen}{\end{center}}
\newcommand{\bbm}{\begin{bmatrix}}
	\newcommand{\ebm}{\end{bmatrix}}
\newcommand{\bmx}{\begin{matrix}}
	\newcommand{\emx}{\end{matrix}}
\newcommand{\bpm}{\begin{pmatrix}}
	\newcommand{\epm}{\end{pmatrix}}
\newcommand{\btab}{\begin{tabular}}
	\newcommand{\etab}{\end{tabular}}

\theoremstyle{plain}
\newtheorem{theorem}{Theorem}[section]

\newtheorem{prop}[theorem]{Proposition}

\newtheorem*{claim*}{Claim}
\newtheorem{thm}[theorem]{Theorem}
\theoremstyle{definition}

\newtheorem{exm}[theorem]{Example}

\numberwithin{equation}{section}
\numberwithin{table}{section}

\def\rn{{\mathbb{R}^n}}
\def\r{{\mathbb{R}}}

\def\a{{\mathcal A}}

\def\n{{\mathbb{N}}}
\def\nd{{\mathbb{N}_d^n}}
\def\n{{\mathbb{N}}}

\def\sm{{\mathbf{S}^d(\mathbb{R}^n)}}
\def\ot{{\otimes d}}

\begin{document}


\title[Moment and Tensor Recovery Problems]
{Moment-SOS Relaxations for Moment and Tensor Recovery Problems}

\author[Lei Huang]{Lei Huang}
\address{Lei Huang, Jiawang Nie, and Jiajia Wang,
Department of Mathematics, University of California San Diego,
9500 Gilman Drive, La Jolla, CA, USA, 92093.}
\email{leh010@ucsd.edu,njw@math.ucsd.edu,jiw133@ucsd.edu}

\author[Jiawang Nie]{Jiawang Nie}

\author[Jiajia Wang]{Jiajia Wang}

\begin{abstract}
This paper studies moment and tensor recovery problems whose decomposing vectors
are contained in some given semialgebraic sets.
We propose Moment-SOS relaxations with generic objectives
for recovering moments and tensors,
whose decomposition lengths are expected to be low.
This kind of problems have broad applications in
various tensor decomposition questions.
Numerical experiments are provided to 
demonstrate the efficiency of this approach.
\end{abstract}

\subjclass[2020]{Primary: 90C23, 15A69; Secondary: 90C22.}
\keywords{moment, tensor, recovery, $K$-decomposable, relaxation}

\maketitle

\section{Introduction}

The moment recovery problem (MRP) studies how to find
a positive measure whose moments satisfy certain conditions.
In applications, people often deal with Borel measures.
For a Borel measure $\mu$ on $\r^n$,
its support is the smallest closed set $T\subseteq \r^n$
such that $\mu(\r^n\text{\textbackslash}T)=0$, for which we denote as $\supp{\mu}$.
Denote by $\nd$ the set of all monomial powers
$\alpha \coloneqq (\alpha_1,\dots,\alpha_n)$
with $|\af|  \coloneqq \af_1 + \cdots + \af_n \le d$.
For $x \coloneqq (x_1,\dots,x_n)$,
denote the monomial $x^{\alpha} \coloneqq x_1^{\alpha_1}\cdots x_n^{\alpha_n}$.
The MRP typically concerns the existence of a Borel measure $\mu$
satisfying linear equations
\begin{equation}   \label{sum1}
\sum_{\alpha\in\nd}a_{i,\alpha} \int x^{\alpha}d\mu
=b_i, \,\, i=1,\ldots, m.
\end{equation}
In the above, the $a_{i,\alpha}$, $b_i$ are given constants.
The integral  $y_{\alpha} \coloneqq  \int x^{\alpha}d\mu$
is called the $\af$th order moment of $\mu$.
The measure $\mu$ is called a representing measure for the multi labelled vector
$y = (y_\af)_{\af \in \nd}$.
If we denote the polynomial
\[
a_i(x)= \sum_{\alpha \in \nd}a_{i,\alpha}x^{\alpha},
\]
then \reff{sum1} is equivalent to
\begin{equation}  \label{sum2}
    \int a_i(x)d\mu  \,= \, b_i, \,\, i=1,\ldots,m.
\end{equation}
In some applications, people may also require that $\mu$ is supported in
a given semialgebraic set $K$ like
\begin{equation} \label{K}
K  \coloneqq
\left\{x \in \mathbb{R}^{n} \left| \begin{array}{l}
		c_{i}(x)=0~(i \in \mathcal{E}), \\
		c_{j}(x) \geq 0~(j \in \mathcal{I})
	\end{array}\right\},\right.
\end{equation}
where $c_i(x)$ and $c_j(x)$ are polynomials in $x$.
The $\mathcal{E}$ and $\mathcal{I}$ are the disjoint label sets
for equality and inequality constraints respectively.
The measure $\mu$ is called $r$-atomic
if $\supp{\mu}$ consists of $r$ distinct points.
In applications, people often prefer the measure $\mu$ satisfying \reff{sum2}
whose support is expected to have small cardinality.
The MRP has broad applications in various moment and tensor decomposition problems \cite{FanZhou17,FNZ19CP,lasserre2008semidefinite,nie2015linear,ZhouFan14,ZhouFan18}.

A typical condition for ensuring the existence of a representing measure
for a truncated multivariate sequence is flat extension
(see \cite{curto2005truncated,Lau05}).
When the support constraining set $K$ is compact,
Moment-SOS relaxations can be used to solve MRPs
(see \cite{dKL19,HKL20,lasserre2008semidefinite,nie2023moment}).
When $K$ is unbounded, the homogenization can be applied to solve
the truncated moment problem
(see \cite{truncated2,huang2023generalized}).
When there is a feasible solution,
one can solve Moment-SOS relaxations to get finitely atomic representing measures
(see \cite{HKL20,nie2023moment}).
When it is infeasible, the nonexistence of a $K$-representing measure can be
detected by solving Moment-SOS relaxations
(see \cite{huang2023generalized,nie2015linear}).
Generally, to solve the MRP is a challenging task.

\subsection*{Tensor recovery}

A $d$th order tensor $\mA \in \mathbb{R}^{n_1 \times \cdots \times n_d}$
can be labeled as the multi-array
\[
\mA = (\mA_{i_1,\dots,i_d}),
1\leq i_1\leq n_1,\dots,1\leq i_d\leq n_d.
\]
For two tensors $\mA$ and $\mB$, their Hilbert-Schmidt inner product is
\[
\langle \mA,\mB \rangle  \, \coloneqq \,  \sum\limits_{1\leq i_j\leq n_j,j=1,\dots, d}
\A_{i_1,\dots,i_d}\mB_{i_1,\dots,i_d}.
\]
When $n_1=\cdots =n_d$, the tensor $\mA$ is said to be symmetric if
$\mA_{i_1,\dots,i_d}=\mA_{j_1,\dots,j_d}$ whenever $(i_1,\dots,i_d)$
is a permutation of $(j_1,\dots,j_d).$
The space of $d$th order symmetric tensors in
$\mathbb{R}^{n \times \cdots \times n}$ is denoted as $\mt{S}^d(\re^n)$.
When $\mA$ is real symmetric,
there always exist $u_1,\dots,u_{r_1},v_1,\dots,v_{r_2} \in \rn$
such that (see \cite{comon2008symmetric})
\begin{equation} \label{1.1}
    \a = (u_1)^{\ot}+\cdots +(u_{r_1})^{\ot}-(v_1)^{\ot}-\cdots -(v_{r_2})^{\ot}.
\end{equation}
In the above, $u^{\otimes d}$ denotes the $d$th tensor power of $u$, i.e.,
\[
(u^{\otimes d})_{i_1,\dots,i_d} =u_{i_1}\cdots u_{i_d}.
\]
The equation \reff{1.1} is called a symmetric decomposition of $\a$.
The smallest $r$ satisfying (\ref{1.1}) is called the  symmetric rank of $\a$.
We refer to \cite{Lim13,GPSTD,LRSTA17}
for recent work on symmetric tensor decompositions.

The tensor recovery problem (TRP) typically concerns whether or not
there exists a tensor satisfying some linear equations.
People often prefer a low rank tensor.
Let $\f_1,\cdots,\f_m \in \mt{S}^d(\re^n)$ be symmetric tensors
and let $b_1,b_2,\ldots,b_m$ be given constants. 
A TRP often concerns the existence of a tensor $\mA \in \mt{S}^d(\re^n)$,
preferably having low real symmetric rank, such that
\begin{equation}\label{1.5}
\langle \f_i,\a\rangle=b_i,\quad \, i = 1,\dots,m.
\end{equation}
In certain applications, people may also require the vectors
$u_i, v_i$ in \reff{1.1} to belong to
a given semialgebraic set $K$ as in (\ref{K}).
Suppose $\a$ admits the decomposition as in \reff{1.1}.
Let $\|\cdot\|$ denote the Euclidean norm and
let $\tilde{u}$ be the normalization of $u$, i.e., $\tilde{u}=\frac{u}{\|u\|}$.
For a point $\tilde{u}$, let $\dt_{\tilde{u}}$ denote the unit Dirac measure
supported at $\tilde{u}$. Denote
 \[
\mu_1 \coloneqq \|u_1\|^d\delta_{\tilde{u}_1}+\cdots+\|u_r\|^d\delta_{\tilde{u}_r},
 \]
 \[
\mu_2 \coloneqq  \|v_1\|^d\delta_{\tilde{v}_1}+\cdots+\|v_r\|^d\delta_{\tilde{v}_r},
 \]
\[
f_i(x) \coloneqq \langle \f_i,x^{\otimes d}\rangle,\,\,  i=1,\dots,m .
\]
One can see that
\[
 \langle \f_i,\a\rangle = \int f_i(x)d\mu_1-\int f_i(x)d\mu_2,\,\, i=1,\dots,m.
\]
Then, \reff{1.5} is equivalent to equivalent to the existence of
$\mu_1$, $\mu_2$ satisfying
\[
\int f_i(x)d\mu_1-\int f_i(x)d\mu_2 = b_i,\,\,i=1,\dots,m.
\]

When $K=\mR^n$, there exist numerical linear algebraic methods,
such as SVD, HOSVD, or T-SVD,  for solving the TRP
(see \cite{svd2,svdbased,existing1,successive,svd3}).
In this paper, we consider TRPs when $K$ is a semialgebraic set as in (\ref{K}).
It contains many difficult tensor decomposition problems,
such as completely positive tensor decompositions \cite{FanZhou17,FNZ19CP}
or positive sum of even powers decompositions \cite{nietensor}.
In some applications, people often prefer the tensor to be recovered
to have small length of decompositions. In view of moment decompositions,
this requires to look for measures $\mu_1$, $\mu_2$
whose supports have small cardinalities.



\subsection*{Contributions}

This paper focuses on solving the generalized truncated moment problem
and the tensor recovery problem.
We propose to solve the following optimization problem
\begin{equation}\label{3.0}
\left\{\begin{array}{ll}
        \min &\int R\, d\mu\\
         \text{ s.t.}
        &\int a_i(x)d\mu = b_i ,\,\, i=1,\dots,m ,\\
        &\mu\in\mathscr{B}(K) .
    \end{array}\right.
\end{equation}
Here, $R$ is a generic sum of squares polynomial and $\mathscr{B}(K)$
denotes the set of Borel measures supported in $K$.
Our major contributions are:

\begin{enumerate}[label=(\roman*)]

\item We propose a Moment-SOS hierarchy to solve (\ref{3.0}).
Typically, the optimal measure of \reff{3.0} has a support of low cardinalities.
This is a reason for choosing the generic $R$.

\item We consider the tensor recovery problem for symmetric tensors
whose decomposing vectors are constrained in the semialgebraic set $K$.
This can be done by solving Moment-SOS relaxations.

\end{enumerate}

The paper is organized as follows.
In Section \ref{section2}, we introduce some preliminaries in polynomial optimization.
Section~\ref{section3} gives the Moment-SOS hierarchy for solving the MRP (\ref{3.0}).
Section~\ref{section4} studies how to solve TRP for positive $K$-decomposable tensors.
In Section~\ref{section5}, we study the recovery problem for general $K$-decomposable tensors.

\section{Preliminaries}
\label{section2}

\subsection*{Notation}

The symbol $\n$ (resp., $\r$) denotes the set of nonnegative integers (resp., real numbers).
For $x= \left(x_{1}, \ldots, x_{n}\right)$ and
$\alpha=\left(\alpha_{1}, \ldots, \alpha_{n}\right) \in \mathbb{N}^{n}$,
we denote $x^{\alpha} \coloneqq x_{1}^{\alpha_{1}} \cdots x_{n}^{\alpha_{n}}$
and $|\alpha| \coloneqq \alpha_{1}+\cdots+\alpha_{n} .$ For an integer $d>0,$
we denote the set of monomial powers by
\begin{equation}
 \begin{aligned} &\quad\mathbb{N}_{d}^{n} \coloneqq \left\{\alpha \in \mathbb{N}^{n}:
0 \leq|\alpha| \leq d\right\}.\\
\end{aligned}
\end{equation}
Let $\mathbb{R}[x] \coloneqq \mathbb{R}\left[x_{1}, \ldots, x_{n}\right]$ be the ring of polynomials in $x \coloneqq \left(x_{1}, \ldots, x_{n}\right)$ with real coefficients and let $\mathbb{R}[x]_{d}$
be the set of polynomials in $\mathbb{R}[x]$ with degrees at most $d$.
The degree of a polynomial $p$ is denoted as $\operatorname{deg}(p)$.
The $[x]_d$ stands for the vector of all monomials in $x$ with degrees at most $d$, i.e.,
\[
[x]_d  \, \coloneqq \,
\begin{bmatrix}
    1&x_1&\cdots&x_n&x_1^2&x_1x_2&\cdots&x_n^d
\end{bmatrix}^T.
\]
For $t \in \mathbb{R},\lceil t\rceil$ (resp., $\lfloor t\rfloor)$
denotes the smallest integer not smaller (resp., the largest integer not bigger) than $t$. For a positive integer $m$, the notation $[m]$ denotes the set $\{1,\dots,m\}$.
For a matrix $A,\, A^{T}$ denotes its transpose. If $A$ is symmetric,
the inequality $A \succeq 0$ (resp., $A \succ 0$)
means that $A$ is positive semidefinite (resp., positive definite).

%
%
In the following, we  review some basics in polynomial optimization.
More detailed introductions can be found in the references \cite{HKL20,lasserre2015introduction,sos,nie2023moment}.
A polynomial $\sig \in \mathbb{R}[x]$ is said to be a sum of squares $(\mathrm{SOS})$
if $\sig =p_{1}^{2}+\cdots+p_{k}^{2}$ for some $p_{1}, \ldots, p_{k} \in \mathbb{R}[x].$
The cone of all SOS polynomials is denoted as $\Sigma[x].$
For a degree $d$, we denote the truncation
\[
\Sigma[x]_{d} \coloneqq  \Sigma[x] \cap \mathbb{R}[x]_{d}.
\]
For a polynomial tuple $h =  (h_{1}, \ldots, h_{m_1} )$,
the ideal generated by $h$ is
\[
\ideal{h} \coloneqq  h_{1} \cdot \mathbb{R}[x]+\cdots+h_{m_1} \cdot \mathbb{R}[x].
\]
The $k$th degree truncation of $\ideal{h}$ is
\begin{equation}
\ideal{h}_{k} \coloneqq h_{1} \cdot \mathbb{R}[x]_{k-\deg\left(h_{1}\right)}+
   \cdots+h_{m_1} \cdot \mathbb{R}[x]_{k-\deg\left(h_{m_1}\right)}.
\end{equation}
For a polynomial tuple $g \coloneqq \left(g_{1}, \ldots, g_{m_2}\right)$,
its quadratic module is
\begin{equation*}
\qmod{g} \coloneqq \Sigma[x]+g_{1} \cdot \Sigma[x]+\cdots+g_{m_2} \cdot \Sigma[x].
\end{equation*}
Similarly, the $k$th degree truncation of $\qmod{g}$ is
\begin{equation}
\qmod{g}_k \coloneqq \Sigma[x]_{k}+g_{1} \cdot \Sigma[x]_{k-\deg(g_1)}+\cdots+g_{m_2} \cdot \Sigma[x]_{k-\deg(g_{m_2})}.
\end{equation}
The set $\ideal{h}+\qmod{g}$ is said to be archimedean if
there exists a scalar $N>0$ such that
\[ N-\|x\|^{2} \in \ideal{h}+\qmod{g} . \]
For the semialgebraic set
\be \label{kk}
S = \left\{x \in \mathbb{R}^{n}\left| \begin{array}{ll}
   h_{i}(x)=0,  & (1 \leq i \leq m_1), \\
   g_{j}(x) \geq 0,\,  & (1 \leq j \leq m_2)\end{array}\right\}, \right.
\ee
if $f \in \ideal{h}+\qmod{g}$, then $f \geq 0$ on $S$.
When $\ideal{h}+\qmod{g}$ is archimedean, if $f > 0$ on $S$,
then we also have $f \in \ideal{h}+\qmod{g}$.
This conclusion is referred to as Putinar's Positivstellensatz.

\begin{thm}\cite{putinar1993positive}\label{putinar}
Suppose $\ideal{h}+\qmod{g}$ is archimedean. If a polynomial $f>0$ on $S$,
then $f \in \ideal{h}+\qmod{g}$.
\end{thm}

For a degree $d$, let $\re^{ \N^n_{d} }$ denote the space of all real vectors $y$
that are labeled by $\af \in \N^n_{d}$. For each
$y \in \re^{ \N^n_{d} }$, we can label it as
$y \, = \, (y_\af)_{ \af \in \N^n_{d} }$.
Such $y$ is called a
{\it truncated multi-sequence} (tms) of degree $d$.
The $y$ is said to admit an $S$-representing measure $\mu$
if there exists a Borel measure $\mu$, with $\supp{\mu} \subseteq S$, such that
\[
y_{\alpha} \,= \, \int x^{\alpha} d\mu, \quad \forall \, \alpha \in \nd.
\]
A tms $y \in \mathbb{R}^{\nd}$  defines the bilinear operation on $ \mathbb{R}[x]_{d}$ as
\begin{equation}
\langle \sum_{\alpha\in\nd} p_{\alpha}x^{\alpha}, y \rangle =
\sum_{\alpha\in\nd} p_{\alpha}y_{\alpha}.
\end{equation}
For a polynomial $q \in \mathbb{R}[x]_{2k}$ and for $y \in \mR^{\mathbb{N}^n_{2k}}$,
the $k$th order localizing matrix of $q$ generated by $y$ is the symmetric matrix
${L}_{q}^{(k)}[y]$ such that
\begin{equation} \label{exp:sdr}
	\left\langle q p^2, y\right\rangle=\operatorname{vec}
\left(p\right)^{T}\left({L}_{q}^{(k)}[y]\right)
\operatorname{vec}\left(p\right),~ \forall \,
p \in \mathbb{R}[x]_{k-\lceil \deg(q)/2\rceil} .
\end{equation}
Here, $\operatorname{vec}\left(p\right)$ denotes the coefficients of $p$
in the graded lexicographical ordering. When $q=1$,
the localizing matrix ${L}_{q}^{(k)}[y]$ reduces to the
$k$th order moment matrix, for which we denote by $M_k[y]$.

Denote the degree
\begin{equation}
    d_S \coloneqq  \displaystyle\max_{i\in[m_1],j\in[m_2]}\{1,\lceil\deg(h_i)/2\rceil,\,
    \lceil\deg(g_j)/2\rceil\}.
\end{equation}
If $y\in\re^{\n^n_{2k}}$ admits an $S$-representing measure $\mu$,
then we must have
\begin{equation}\label{flat1}
M_k[y]\succeq 0, \quad
L_{h_i}^{(k)}[y]=0 \,(i \in [m_1]),\quad
L_{g_j}^{(k)}[y]\succeq 0\,(j \in [m_2]).
\end{equation}
A typical condition for $y$ to admit an $S$-representing measure
is the flat extension or flat truncation
\cite{curto2005truncated,niecertificate}.
The tms $y$ is said to have a  flat truncation if
there exists a degree $ t\in [d_S,k]$ such that
\be \label{flat:tru}
\rank \, M_{t-d_S}[y] \, = \, \rank \, M_t[y].
\ee
When \reff{flat1} and \reff{flat:tru} hold, the truncation
\be \label{y2t}
y|_{2t}  \, \coloneqq \, (y_\af)_{ \af \in \N^n_{2t} }
\ee
admits an $S$-representing measure whose support consists of
$\rank \, M_t[y]$ points contained in $S$.
We refer to \cite{HKL20,lasserre2015introduction,sos,nie2023moment}
for these facts.

\section{The Moment Recovery Problem}
\label{section3}

For a polynomial $a_i(x)=  \sum_{\alpha \in \nd}a_{i,\alpha}x^{\alpha} \in \re[x]_d$,
if a tms $y \in \re^{\N^n_d}$ admits a representing measure $\mu$, then it holds
\[
 \int a_i(x)d\mu=\sum_{\alpha\in\nd}a_{i,\alpha}y_{\alpha}=\langle a_i,y\rangle.
\]
Let $K$ be the semialgebraic set as in \reff{K} and let
\[
d_1  =  2 \lceil d/2 \rceil .
\]
The $d_1$th order moment cone of the set $K$ is
\be \label{set:Rd(K)}
\mathscr{R}_{d_1}(K) \,  \coloneqq \,
\Bigg \{  y  \in \re^{ \N^n_{d_1} }
\left| \baray{c}
\exists \,  \mu \in \mathscr{B}(K) \, \st \\
y_{\alpha}=\int x^{\alpha} \mt{~d}\mu \, \, \forall  \,  \af \in \N^n_{d_1}
\earay \right.
\Bigg \}.
\ee
Choose a generic SOS polynomial $R \in \Sig[x]_{d_1}$, then
the optimization problem \reff{3.0} is equivalent to
\begin{equation}\label{p}
\left\{\begin{array}{cl}
    \min &\langle R,y\rangle \\
    \st &\langle a_i,y\rangle = b_i\,(i = 1,\dots,m),\\
    &y\in\mathscr{R}_{d_1}(K).
\end{array}\right.
\end{equation}
Note that the tms variable $y$ in \reff{p} has degree $d_1$, instead of $d$.
Denote the degree-$d_1$ nonnegative polynomial cone
\[
\mathscr{P}_{d_1}(K)  \coloneqq   \{f\in\r[x]_{d_1}:\,\,f(x)\ge 0 \,\, \forall x\in K \}.
\]
The cone $\mathscr{P}_{d_1}(K)$ is dual to $\mathscr{R}_{d_1}(K)$, i.e.,
 $\langle f, y \rangle \ge 0$
for all $f \in \mathscr{P}_{d_1}(K)$ and for all $y \in \mathscr{R}_{d_1}(K)$.
For $\lambda = (\lambda_1,\dots,\lambda_m)$,
the dual optimization of \reff{p} is
\begin{equation}\label{d}
\left\{\begin{array}{cl}
 \max  &  b^T\lambda \\
\st & R-\displaystyle\sum^{m}_{i=1}\lambda_i a_i \in \mathscr{P}_{d_1}(K),
\end{array}\right.
\end{equation}
Denote by $\phi^*$, $\psi^*$ the optimal values of \reff{p}, \reff{d} respectively.

For a degree $k  \ge  d_1/2$,
the $k$th order moment relaxation of {(\ref{p})} is
\begin{equation}\label{pk}
\left\{\begin{array}{cl}
 \min &\langle R,w\rangle\\
 \text{ s.t. }&\langle a_i,w\rangle = b_i\,\,\,(i=1,\dots,m),\\
 &L_{c_i}^{(k)}[w]=0\,\,(i\in\mathcal{E}),\\
 &L_{c_j}^{(k)}[w]\succeq0~(i\in\mathcal{I}),\\
 &M_k[w]\succeq 0, \\
 & w \in \re^{ \N^n_{2k} }.
 \end{array}\right.
\end{equation}
The tms variable $w$ in the above has degree $2k$.
The dual optimization of {(\ref{pk})} is  the $k$th order SOS relaxation
\begin{equation}  \label{dk}
\left\{\begin{array}{cl}
    \max & b^T\lambda\\
    \st& R-\displaystyle\sum^m_{i=1}\lambda_ia_i \in
    \ideal{c_{eq}}_{2k}+\qmod{c_{in}}_{2k}.
\end{array}\right.
\end{equation}
In the above, the polynomial tuples $c_{eq}$, $c_{in}$ are
\[
c_{eq}=(c_i)_{i\in\mathcal{E}}, \quad c_{in} = (c_j)_{j\in\mathcal{I}} .
\]
Denote by $\phi_k$, $\psi_k$ the optimal values of (\ref{pk}), (\ref{dk}) respectively.
%
%
The following are some properties of the above optimization problems.

\begin{prop} \label{thm31}
Suppose $K$ is compact and \reff{p} is feasible.
\begin{enumerate}[label = (\roman*)]
\item The optimal values of {(\ref{p})} and {(\ref{d})} are equal,
and the optimal value of {(\ref{p})} is achievable.

\item If $R$ is generic in $\Sigma[x]_{d_1}$, then \reff{p} has a unique minimizer $y^{\ast}$,
and $y^\ast$ admits a $K$-representing measure $\mu$ that is $r$-atomic with $r\le m$.
\end{enumerate}
\end{prop}
\begin{proof}
(i) Note that $\lambda = 0$ is an interior point of {(\ref{d})}, since $K$ is compact.
By the strong duality theorem for linear conic optimization theory
(e.g., \cite[Theorem 1.5.1]{nie2023moment}), the optimization problem
\reff{p} and \reff{d} have the same optimal value
and it is achievable for \reff{p}.

(ii) Since $K$ is compact, it follows from Proposition 5.2 of \cite{nie2014truncated}
that {(\ref{p})} has a unique minimizer $y^{\ast}$, for generic $R \in \Sigma[x]_{d_1}$.
In fact, $y^{\ast}$ must be  an extreme point of the feasible set of \reff{p}.
Since $y^\ast \in\mathscr{R}_{d_1}(K)$, it must admit a $K$-representing measure $\mu$
that is $r$-atomic with $r\le m$, by Carath\'{e}odory's Theorem.
\end{proof}

The following is the convergence theorem of the hierarchy of relaxations
\reff{pk}-\reff{dk}. We refer to the book \cite{nie2023moment}
for the linear independence constraint qualification condition (LICQC),
strict complementarity condition (SCC), and second order sufficient condition (SOSC)
in nonlinear optimization.

\begin{thm}\label{thm316}
Suppose $R$ is generic in $\Sigma[x]_{d_1}$, optimization {(\ref{p})} is feasible and  $\ideal{c_{eq}}+\qmod{c_{in}}$ is archimedean.
Suppose $\lambda^{\ast} = (\lmd_1^\ast, \ldots, \lmd_m^\ast)$
is a maximizer of \reff{d} and $w^{(k)}$ is a minimizer of \reff{pk}. Then, we have:

\begin{enumerate}[label=(\roman*)]

\item If   $w^{(k)}_0 = 0$,
then $w^{(k)}$ has a flat truncation.

\item If $w^{(k)}_0 \neq 0$ and LICQC, SCC, and SOSC all hold at every minimizer of
the polynomial optimization problem
\begin{equation}\label{dloc}
\left\{ \baray{cl}
 	\min\limits_{x \in \re^n} & R(x)-\sum\limits^{m}_{i=1}\lambda_i^* a_i(x) \\
 	 \text{s.t.} & c_i(x)=0, ~(i\in \mathcal{E}),\\
  &c_j(x)\ge0, ~(j\in \mathcal{I}),
\earay\right.
\end{equation}
then $\phi_{k}=\psi_k=\phi^{\ast}$ and every minimizer of \reff{pk}
has a flat truncation, when $k$ is sufficiently large.

\end{enumerate}
\end{thm}

\begin{proof}
(i) When $w^{(k)}_0 = 0$, we have $w_{\alpha}^{(k)}=0$ for all $|\alpha|\le k$,
since $M_k[w^{(k)}]\succeq 0$.  This implies that $M_k[w^{(k)}] \operatorname{vec}(x^{\alpha})=0$
for all $|\alpha| \le k-1$. For every $|\alpha|\le 2k-2$, we can rewrite  it as $\alpha = \beta+\eta$ with
$|\beta| \le k-1$ and $|\eta|\le k-1$, then we have that
\[
w_{\alpha}^{(k)}=\operatorname{vec}(x^{\beta})^T M_k[w^{(k)}]\operatorname{vec}(x^{\eta})=0.
\]
Therefore, the truncation $w^{(k)}|_{2k-2}$ is flat.

(ii) We first show that $\psi_k = \psi^{\ast}$ for all $k$ big enough.
For convenience, denote
\[
R_{\lambda} \, \coloneqq \, R-\sum^m_{i=1}\lambda_ia_i .
\]
Clearly, the minimum value of \reff{dloc} is nonnegative.
When LICQC, SCC and SOSC hold at every minimizer,
Theorem 1.1 of \cite{niecondition} implies that for  $k$ big enough, there exists $\sigma_{0}\in\qmod{c_{in}}$ such that
\[
R_{\lambda^{\ast}} \in \sigma_{0}+\ideal{V_{\r}(c_{eq})} .
\]
It means that  $R_{\lambda^{\ast}}-\sigma_0$
vanishes identically on the real variety $V_{\r}(c_{eq}).$
Since $\ideal{c_{eq}}+\qmod{c_{in}}$ is archimedean, the set $K$ is compact.
For generic $R \in \Sigma[x]_{d_1}$, we have $R >0$ on $K$.
By Theorem 3.2 of \cite{hny2}, one can similarly show that there exists $k_0\in\n$
such that for every $\eps > 0$, it holds that
\[
R_{\lambda^{\ast}} + \eps R \in  \ideal{c_{eq}}_{2k_0}+\qmod{c_{in}}_{2k_0}.
\]
Therefore, for every $\eps > 0$, $\lambda^{\ast}/(1+\eps)$ is feasible for \reff{dk}
at the relaxation order $k_0$. As $\eps \to 0$, we get
\[
b^T \lambda^{\ast}/(1+\eps) \, \rightarrow \, b^T\lambda^{\ast}.
\]
This means that $\psi_{k_0} = \phi^{\ast}$. Moreover, since $\psi_{k_0} \le \psi_k \le \phi^*$
for all $k \ge k_0$, we get $\psi_k = \psi^{\ast}$ for all $k \ge k_0$. By the strong duality theorem, we have $\phi_{k}=\psi_k=\phi^{\ast}$.
It follows from  Theorem 3.2 \cite{hny2} that  every minimizer $w^{(k)}$ of \reff{pk}
has a flat truncation when $k$ is big enough.
\end{proof}

%
%
In the following,  we give some examples to
show the performance of relaxations \reff{pk}--\reff{dk}
for solving \reff{p}-\reff{d}.
We choose the generic SOS polynomial $R$ as
\[
R=[x]_{d_1/2}^T G^T G[x]_{d_1/2},
\]
where $G$ is a randomly generated square matrix
whose entries obey the normal distribution.
The relaxations \reff{pk}--\reff{dk} are solved by the software
{\tt GloptiPoly~3} \cite{gloptipoly},
which calls the SDP solver {\tt SeDuMi} \cite{sedumi}. The computation is implemented in MATLAB R2023b on a laptop with  16G RAM.
For cleanness, only four decimal digits are shown
for computational results.

\begin{exm} Consider the polynomials
\[
\begin{gathered}
a_1 = x_1^3x_2-x_1^2x_2^2,\quad a_2 = x_2^3x_3-x_2^2x_3^2,\quad
a_3 = x_1^4-x_2^4,\\
a_4 = x_3^3x_4-x_3^2x_4^2,\quad a_5 = x_4^3x_1-x_4^2x_1^2,\quad
a_6= x_3^4-x_4^4,
\end{gathered}
\]
and
$b = \begin{bmatrix}
    1&1&2&1&1&2
\end{bmatrix}^T.
$
The set $K$ is
\[
 K = \{x\in\r^4\,: \,x\ge 0,\,x_1x_2-x_2x_3=0,\,x_2x_3-x_1x_4=0,\,\|x\|^2=1\}.
\]
For $k=2$, solving \reff{pk}, we get the $2$-atomic measure
$\mu =   \lambda_1\delta_{u_1} + \lambda_2\delta_{u_2}$, where
\[
\begin{gathered}
u_1= \begin{bmatrix}
    0.6920&
    0.1453&
    0.6920&
    0.1453
\end{bmatrix}^T,\quad \lambda_1 =   36.0156, \\
u_2=\begin{bmatrix}
     0.1954&
    0.6796&
    0.1954&
    0.6796
\end{bmatrix}^T,\quad \lambda_2 =  29.4743.
\end{gathered}
\]
The computation took around 0.13s.
\end{exm}

\begin{exm}\label{examplerandom}(random instances)
We consider some randomly generated MRPs. Let $K$ be the unit sphere
\[
K = \{x\in\r^n:\,\|x\|^2= 1\}.
\]
For a degree $d$, we generate a tms $y$ as
\[
y \, = \, c_1[u_1]_{d}+\dots+c_N[u_N]_{d}, \quad c_i > 0,
\]
where  $N=\binom{n+d}{d}$ and all entries of $u_i$
obey the normal distribution. The moment equations are
\[
b_i =\langle a_i,y\rangle,\,i=1,\dots,m,
\]
for randomly generated polynomials $a_i \in \re[x]_d$.
The coefficients of $a_i$ also obey the normal distribution.
For each case of $(n,d)$, we generate $100$ random instances and solve \reff{pk}.
Let $r$ be the cardinality of the representing measure for
a flat truncation of the minimizer of \reff{pk}.
The obtained length $r$ is reported in Table~\ref{Table:01}.
One can observe that  $r$ is typically much less than $m$.
\begin{table}[htb]
\centering
\caption{Values of the length $r$ for different  $(n,d)$.}
\label{Table:01}
\begin{tabular}{|c|c|c|c|c|}  \hline
 $(n,d)$ & $m = 4$ & $m = 6$ & $m = 8$ & $m = 10$ \\
 \hline
  $(3,3)$  & $1,2,3$   & $1,2,3,4$ & $1,2,3,4,5$  & $2,3,4,5,6$ \\
\hline
$(5,4)$ & $1,2,3$   & $1,2,3,4$ & $2,3,4,5,6$  & $1,3,4,5,6$  \\
\hline
$(8,6)$ & $1,2,3$   & $1,2,3$ & $1,2, 3, 4$  & $2,3,5,6$  \\
\hline
$(10,8)$ & $1,2,3$  & $1,2,3,5$&$2,3,4,5,6$  & $2,3,4,5,6$ \\
\hline
\end{tabular}
\end{table}
\end{exm}

\section{Recovery for positive $K$-decomposable tensors}
\label{section4}

For positive integers $n$ and $d$, recall that $\mt{S}^d(\re^n)$ denotes the space of $n$-dimensional real symmetric tensors of order $d$.
In this section, we assume the set $K$ is given as
\begin{equation}\label{K2}
K  =  \left\{
\begin{array}{l|l}x \in \mathbb{R}^n &
\begin{array}{l}c_i(x)=0\,(i \in \mathcal{E}), \\
c_j(x) \geq 0\,(j \in \mathcal{I}),\\
 \|x\|^2=1
\end{array}\end{array}\right\},
\end{equation}
where $c_i$ and $c_j$ are homogeneous polynomials in $x$.
A tensor $\a\in\mt{S}^d(\r^n)$ is said to be {\it positive $K$-decomposable}
if there exist vectors $u_1,\dots,u_r\in K$ such that
\[
\a = \sum^r_{i=1}\lambda_i(u_i)^{\otimes d}, \quad  \lambda_i > 0 .
\]
In some applications, people may require that the tensor $\a$
satisfies some linear equations like
\begin{equation}\label{sec4:equ}
\langle \f_i,\a\rangle=b_i,\, \,  i = 1,\dots,m,
\end{equation}
for given $\f_i \in \mt{S}^d(\r^n)$ and $b_i \in \re$.

Each $\a \in \mt{S}^d(\r^n)$ can be identified as
a homogeneous truncated multi-sequence (htms)
$h = (h_\af) \in \mathbb{R}^{\overline{\mathbb{N}}_{d}^{n}}$
in the way that (note $\af = (\af_1, \ldots, \af_n)$)
\begin{equation}\label{y}
 h_{\alpha}=\mathcal{A}_{\sigma_1 \dots \sigma_m} ~\quad \text{for}~ x_1^{\alpha_1}\cdots x_n^{\alpha_n}=x_{\sigma_1}\cdots x_{\sigma_m},
\end{equation}
where the power set $\overline{\mathbb{N}}_{d}^{n}$ is
\[
\overline{\mathbb{N}}_{d}^{n} \, \coloneqq \,
\left\{\alpha \in \mathbb{N}^{n}:|\alpha|  = d\right\}.
\]
Denote the homogeneous polynomials
\[
f_i(x)=\langle \f_i,x^{\otimes d}\rangle, \,\,   i=1, \ldots, m.
\]
Then $\langle \f_i, \a\rangle=\langle f_i, h \rangle$ for each $i$.
The tensor $\a$ satisfies \reff{sec4:equ}
if and only if the htms $h$ satisfies the equations
\begin{equation}\label{ff}
\langle f_1,h\rangle = b_1,\dots,\langle f_m,h\rangle = b_m.
\end{equation}

Let $d_1 = 2 \lceil d/2 \rceil$. We consider the moment optimization problem
\begin{equation}\label{tp}
\left\{\begin{array}{cl}
\min&\langle R, z\rangle \\
\text{s.t.} &\langle f_i,z \rangle = b_i\,(i = 1,\dots,m),
\\&z \in \mathscr{R}_{d_1}(K),
    \end{array}\right.
\end{equation}
where $R \in \Sig[x]_{d_1}$ is a generic SOS polynomial.
The above tms variable $z$ can be viewed as an extension of the htms $h$.
Note that  \reff{tp} is a moment recovery problem as in Section \ref{section3}.
The dual optimization problem of \reff{tp} is
\begin{equation} \label{td}
\left\{\begin{array}{cl}
  \max   &  b^T\lambda \\
  \text{s.t.} & R-\displaystyle\sum^{m}_{i=1}\lambda_i f_i \in \mathscr{P}_{d_1}(K) .
\end{array}\right.
\end{equation}
For a degree $k$, the $k$th order moment relaxation of \reff{tp} is
\begin{equation}\label{tpk}
\left\{\begin{array}{cl}
    \min&\langle R,w\rangle\\
    \text{s.t.}&\langle f_i,w\rangle = b_i\,\,\,(i=1,\dots,m),\\
    &L_{c_i}^{(k)}[w]=0\,\,(i\in\mathcal{E}),\\
    &L_{c_j}^{(k)}[w]\succeq0\,\,(j\in\mathcal{I}),\\
    &L_{\|x\|^2-1}^{(k)}[w]=0,\\
    &M_k[w]\succeq 0, \\
    & w \in \re^{ \N^n_{2k} }.
\end{array}\right.
\end{equation}
The dual optimization of {(\ref{tpk})} is the $k$th SOS relaxation
\begin{equation}\label{tdk}
    \left\{\begin{array}{cl}
        \max
        &b^T\lambda\\
        \text{s.t.}&R-\displaystyle\sum^m_{i=1}\lambda_ia_i\in \ideal{c_{eq},\,\|x\|^2-1}_{2k}+\qmod{c_{in}}_{2k}.
    \end{array}\right.
\end{equation}

The convergent analysis of \reff{tpk}--\reff{tdk}
is quite similar to that in Section \ref{section3}.
We omit it for cleanness of presentation.
Suppose $w^{(k)}$ is a minimizer of \reff{tpk} and it has a flat truncation:
there exists a degree $t \in [d_1/2, k]$ such that
\be
\begin{gathered}
\rank \, M_{t-d_k}[w^{(k)}] \, = \, \rank \, M_{t}[w^{(k)}], 
\end{gathered}
\ee
where the degree $d_K = \max_{i \in \mc{E} \cup \mc{I} }
\big \{1, \lceil \deg(c_i) /2 \rceil \big \}$.
Then, there exist points $u_1,\dots,u_r \in K$  such that
\[
w^{(k)}|_{2t} = \sum\limits_{i=1}^r\lambda_{i}\left[u_{i}\right]_{2t}, \, \lmd_i >  0.
\]
It then implies the tensor decomposition (note $2t \ge d$)
\[
\mathcal{A} \, = \,  \lambda_{1}\left(u_{1}\right)^{\otimes d}+\cdots+\lambda_{r}\left(u_{r}\right)^{\otimes d}.
\]
%
%
The following are some examples to show the performance
of relaxations \reff{tpk}-\reff{tdk}.

\begin{exm}\label{example4.4}

Consider the set
$$
K = \{x\in\mathbb{R}^4: x_1+x_2+x_3+x_4\ge 0,\,\|x\|^2=1\}.
$$
We look for a positive $K$-decomposable tensor
$\mathcal{A} \in  \mathbf{S}^3(\mathbb{R}^4)$ satisfying
\[
\begin{gathered}
\a_{111}-2\a_{222}=1,\quad2\a_{222}-3\a_{333}=1, \\
3\a_{333}-4\a_{444}=1,\quad2\a_{234}-3\a_{111}=1,\\
2\a_{133}-\a_{111}=1,\quad 2\a_{123}-\a_{244}=1,\\
2\a_{123}-\a_{333}=1,\quad \a_{123}-3\a_{344}=1.
\end{gathered}
\]
For $k = 3$, solving \reff{tpk}, we get
$\a = \lambda_1(u_1)^{\otimes 3}+\lambda_2(u_2)^{\otimes 3}+\lambda_3(u_3)^{\otimes 3}+\lambda_4(u_4)^{\otimes 3}$ with
\[
\begin{array}{ll}
  u_1 = \begin{bmatrix}
      0.3739&
   -0.3938&
   -0.5837&
    0.6036
\end{bmatrix}^T,   & \lambda_1 =   4.0168, \\ 
u_2 = \begin{bmatrix}
      -0.8588&
    0.1844&
    0.3594&
    0.3151
\end{bmatrix}^T,
     & \lambda_2 =   1.2445.\\
     u_3 = \begin{bmatrix}
   0.4349&
   -0.7627&
    0.4600&
   -0.1322
\end{bmatrix}^T,
     & \lambda_3 =    0.4925.\\
     u_4 = \begin{bmatrix}
     0.5668&
   -0.0367&
    0.2531&
   -0.7832
\end{bmatrix}^T,
     & \lambda_4 =    3.4336.
\end{array}
\]
The computation took around 0.38s.
\end{exm}

A tensor $\mA \in\mt{S}^d(\mathbb{R}^m)$ is said to be completely positive (CP)
if it is positive $K$-decomposable for the set
$K=\{x\in \mr^n: x_1\geq 0,\dots,x_n \geq 0,\,\|x\|^2=1\}$.

\begin{exm}
We look for a completely positive tensor $\mA \in  \mathbf{S}^3(\mathbb{R}^4)$ satisfying
\[
\begin{gathered}
\mathcal{A}_{113}+\mathcal{A}_{123}-\mathcal{A}_{223}=2,\quad \mathcal{A}_{223}+\mathcal{A}_{333}-\mathcal{A}_{114}=10, \\
\mathcal{A}_{333}+\mathcal{A}_{124}+\mathcal{A}_{134}=8,\quad \mathcal{A}_{114}-\mathcal{A}_{224}-\mathcal{A}_{234}=1.
\end{gathered}
\]
For the relaxation order $k = 2$, we get that
$\a = \lambda_1(u_1)^{\otimes 3}+\lambda_2(u_2)^{\otimes 3}$, where
\[
\begin{array}{ll}
  u_1 = \begin{bmatrix}
         0.8165&
    0.0000&
    0.0000&
    0.5774
    \end{bmatrix}^T,   &  \lambda_1 =2.5958,
    \\
    u_2 = \begin{bmatrix}
        0.4193&
    0.4741&
    0.7742&
   0.0000
    \end{bmatrix}^T,
     & \lambda_2 = 17.2414.
\end{array}
\]
The computation took around 0.30s.
\end{exm}

\section{General tensor recovery}
\label{section5}

In this section, we discuss general tensor recovery.
Let $K$ be the  semialgebraic set as in (\ref{K2}).
We look for a symmetric tensor $\a \in \mt{S}^d(\re^n)$ satisfying equations
\begin{equation}\label{1.5s}
\langle \f_i,\a\rangle=b_i,\,\, i = 1,\dots,m .
\end{equation}
Typically, we are interested in $\mA$ having the decomposition
\be \label{sytK}
\a = \sum^r_{i=1}\lambda_i(u_i)^{\otimes d}, \quad u_i \in K, \quad \lmd_i \in \re.
\ee
If the above hold, the tensor $\mA$ is said to be $K$-decomposable.
It is interesting to note that \reff{sytK} can be rewritten as
\[
\mathcal{A}=\sum\limits_{\lambda_i\geq 0} \lambda_{i}
 \left( u_{i}\right)^{\otimes d}-\sum\limits_{\lambda_i\leq 0}
 (-\lambda_{i})\left( u_{i}\right)^{\otimes d}.
\]
Let $h$ be the htms of $\a$ as in \reff{y}.
If we denote the measures
\begin{equation*}
\mu_1=\sum\limits_{\lambda_i\ge 0} \lambda_{i}\delta_{{u_i}},~\mu_2=\sum\limits_{\lambda_i\le 0} (-\lambda_{i})\delta_{{u_i}},
\end{equation*}
then for each $\alpha \in \overline{\mathbb{N}}_{d}^{n}$,
\[
h_{\alpha}=\int x^{\alpha} d\mu_1-\int x^{\alpha} d\mu_2.
\]
For the polynomials $f_i(x)=\langle \f_i,x^{\otimes d}\rangle$,
the relation~\reff{1.5s} becomes
\begin{equation} \label{fiA =bi}
\langle \f_i,\a\rangle=\int f(x) d\mu_1-\int f(x) d\mu_2=b_i,\, \, i = 1,\dots,m.
\end{equation}

In order to recover a tensor $\mA$ having the decomposition \reff{sytK} from
equations \reff{1.5s}, we consider the optimization problem
\begin{equation}\label{4.8}
\left\{\begin{array}{cl}
\min & \int R_1\, \mathrm{d} \mu_1+\int R_2\, \mathrm{d} \mu_2 \\
\hfill\text{s.t.} & \int f_i(x) \mathrm{d} \mu_1-
  \int f_i(x) \mathrm{d} \mu_2 = b_i\,\,(i=1,\dots,m), \\
 & \mu_1, \mu_2 \in \mathscr{B}\left(K\right).
\end{array}\right.
\end{equation}
Let $d_1 = 2 \lceil d/2 \rceil$.
If we use moments, the above optimization is equivalent to
\begin{equation} \label{4.9}
\left\{\begin{array}{cl}
\min & \langle R_1,y\rangle + \langle R_2,z\rangle \\
\st  & \langle f_i, y\rangle -\langle f_i, z\rangle =b_i\,\,(i=1,\dots,m), \\
 & y, z \in \mathscr{R}_{d_1}(K).
\end{array}\right.
\end{equation}
We remark that if we select $R_1 = R_2 = 1$ (the constant one polynomials),
the optimal value of \reff{4.9} will be the smallest nuclear norm of
$\mA$ satisfying \reff{fiA =bi}, when $K$ is the unit sphere.
We refer to \cite{nietensor} for tensor nuclear norms.
In this paper, we try to recover $\mA$ such that
its decomposition length is small.
For this purpose, we select generic SOS polynomials $R_1, R_2 \in \Sig[x]_{d_1}$.
The dual optimization problem of {$(\ref{4.9})$} is
\begin{equation}\label{4.10}
\left\{\begin{array}{cl}
\max  &  b^T\lambda \\
\st &  R_1- \displaystyle \sum^{m}_{i=1}\lambda_if_i \in \mathscr{P}_{d_1}(K),\\
&\displaystyle R_2+\sum^{m}_{i=1}\lambda_if_i\in \mathscr{P}_{d_1}(K).
\end{array}\right.
\end{equation}
Denote the optimal value of {$(\ref{4.9})$}, {$(\ref{4.10})$} by $\phi^{\ast}$, $\psi^{\ast}$. For a degree $k$, the $k$th order moment relaxation of {$(\ref{4.9})$} is
\begin{equation}\label{4.11}
\left\{\begin{array}{cl}
\min &\langle R_1,w\rangle +\langle R_2,v\rangle \\
\text{s.t.}
& \langle f_i, w\rangle -\langle f_i, v\rangle =b_i\,\,(i=1,\dots,m),  \\
& L_{c_i}^{(k)}[w]=0\,\,(i\in\mathcal{E}),\, L_{c_i}^{(k)}[v]=0\,\,(i\in\mathcal{E}),\\
& L_{c_j}^{(k)}[w]\succeq 0\,\,(i\in\mathcal{I}),\,L_{c_j}^{(k)}[v]\succeq 0\,\,(i\in\mathcal{I}),\\
& L_{\|x\|^2-1}^{(k)}[w]=0,\,L_{\|x\|^2-1}^{(k)}[v]=0,\\
& M_k[w]\succeq 0,\, M_k[v]\succeq 0,\\
& w,v \in \re^{ \N^n_{2k} }.
\end{array}\right.
\end{equation}
The dual optimization of {$(\ref{4.11})$} is then
\begin{equation}\label{4.12}
\left\{\begin{array}{cl}
\max &  b^T\lambda \\
\text{s.t.}& R_1- \displaystyle \sum^{m}_{i=1}\lambda_if_i \in \ideal{c_{eq},\|x\|^2-1}_{2k}+ \qmod{c_{in}}_{2k},\\
&R_2+ \displaystyle \sum^{m}_{i=1}\lambda_if_i \in \ideal{c_{eq},\|x\|^2-1}_{2k} + \qmod{c_{in}}_{2k}.
\end{array}\right.
\end{equation}
Denote the optimal values of \reff{4.11}, \reff{4.12} by $\phi_{k},\psi_k$, respectively.
Suppose $(w^{(k)},v^{(k)})$ is a minimizer of \reff{4.11} and 
both $w^{(k)}$ and $v^{(k)}$ have flat truncations:
there exists a degree $t \in [d_1/2, k]$ such that
\be 
\begin{gathered}
\rank \, M_{t-d_k}[w^{(k)}] \, = \, \rank \, M_{t}[w^{(k)}],  \\
\rank \, M_{t-d_k}[v^{(k)}] \, = \, \rank \, M_{t}[v^{(k)}]. 
\end{gathered}
\ee
where the degree $d_K = \max_{i \in \mc{E} \cup \mc{I} } 
\big \{1, \lceil \deg(c_i) /2 \rceil \big \}$.
Then, there exist vectors $u_1,\dots,u_{r}\in K$ and $r_1 \le r$ such that
\[
w^{(k)}=\sum_{i=1}^{r_1} \lambda_{i}\left[u_{i}\right]_{2k}, \, \lmd_i > 0,  \quad v^{(k)}=\sum_{j=r_1+1}^{r} \lmd_{j}\left[u_{j}\right]_{2k}, \, \lmd_j > 0.
\]
The above implies the tensor decomposition  (note $2t \ge d$)
\[
\mathcal{A} = \sum_{i=1}^{r_1} \lambda_{i}\left(u_{i}\right)^{\otimes d} -
\sum_{j=r_1+1}^{r} \lmd_{j} \left( u_j \right)^{\otimes d}.
\]

We refer to \cite{nie2023moment} for the optimality conditions
LICQC, SCC and SOSC.

\begin{thm}
Suppose $\lambda^* = (\lmd_1^*, \ldots, \lmd_m^*)$ is a maximizer of \reff{4.10}
and $(w^{(k)},v^{(k)})$ is a minimizer of \reff{4.11}.
Consider the following optimization problems
 \begin{equation}\label{dloc2}
   \left\{ \baray{cl}
   	\min\limits_{x \in \re^n}  & R_1(x) - \sum\limits^{m}_{i=1}\lambda_i^* f_i(x) \\
    \st & c_i(x)=0 ~(i\in \mathcal{E}),\\
     &c_j(x)\ge 0 ~(j\in \mathcal{I}),\\
      &\|x\|^2-1=0,\\
    \earay\right.
\end{equation}
and
\begin{equation}\label{dloc3}
   \left\{ \baray{cl}
    \min\limits_{x \in \re^n} & R_2(x) + \sum\limits^{m}_{i=1}\lambda_i^* f_i(x) \\
    \st & c_i(x)=0 ~(i\in \mathcal{E}),\\
     &c_j(x)\ge 0 ~(j\in \mathcal{I}),\\
     &\|x\|^2-1=0.\\
    \earay\right.
\end{equation}
Then, when $k$ is big enough, we have $\phi_{k}=\psi_{k}=\phi^{\ast}$
and the  minimizer $(w^{(k)},v^{(k)})$ of \reff{4.11} has a flat truncation, if any one of the following holds:
\begin{enumerate}[label=(\roman*)]
\item $w^{(k)}_0>0,\,v^{(k)}_0>0$ and LICQC, SCC, SOSC hold at every minimizer of \reff{dloc2} and  \reff{dloc3}.

\item $w^{(k)}_0=v^{(k)}_0=0$.

\item $w^{(k)}_0=0,\,v^{(k)}_0>0$ and LICQC, SCC, SOSC hold at every minimizer of \reff{dloc3}.

\item $w^{(k)}_0>0,\,v^{(k)}_0=0$ and LICQC, SCC, SOSC hold at every minimizer of \reff{dloc2}.
\end{enumerate}

\end{thm}

\begin{proof}
(i)	
Since $R_1, R_2$ are generic in $\Sig[x]_{d_1}$, $\lambda = (0,\dots ,0)$ is an interior point of
the feasible set of \reff{4.10}. Thus, the Slater condition holds
for \reff{4.8}-\reff{4.9}, and hence $\phi^{\ast} = \psi^{\ast}$.
It follows from the proof of
Theorem~\ref{thm316} of \cite{hny2} that there exists $k_0\in\n$
such that for all $\epsilon>0$,
\[
R_1-\sum_{i=1}^{m}\lambda^{\ast}_if_i(x)+\epsilon R_1
 \in\ideal{c_{eq},\|x\|^2-1}_{2k_0}+\qmod{c_{in}}_{2k_0},
\]
\[
R_2+\sum_{i=1}^{m}\lambda^{\ast}_if_i(x)+\epsilon R_2
\in\ideal{c_{eq},\|x\|^2-1}_{2k_0}+\qmod{c_{in}}_{2k_0}.
\]
It implies that $\lmd^*/(1+\eps)$ is feasible for \reff{4.12}.
As $\epsilon \to 0$, we get
\[
b^T(\frac{\lmd^\ast}{1+\epsilon})=\frac{1}{1+\epsilon}b^T\lambda^{\ast}\rightarrow b^T\lambda^{\ast}.
\]
Therefore, $\psi_{k_0} = \psi^{\ast}$, and by the strong duality, we have
$\phi_{k_0} = \psi_{k_0} = \phi^{\ast}$. Then, we have
$\phi_{k} = \psi_{k} = \phi^{\ast}$ for all $k \ge k_0$.
The conclusion that every minimizer of \reff{4.11}
has a flat truncation follows from Theorem~5.1 of \cite{hny2}.

(ii) When $w_0^{(k)} = v_0^{(k)} = 0$, the constraints $M_k[w^{(k)}]\succeq 0$ and $M_k[v^{(k)}]\succeq 0$ imply that $w_{\alpha}^{(k)}=0$ and $v_{\alpha}^{(k)}=0$ for all $\alpha$ such that $|\alpha|\le k.$ Similarly as in Theorem \ref{thm316} (i), we have that
 the truncations $w^{(k)}|_{2k-2}=0$, $v^{(k)}|_{2k-2}=0$ and then $w^{(k)}|_{2k-2}$ and $v^{(k)}|_{2k-2}$ are flat.
 Thus, $(w^{(k)}, v^{(k)})$ is a minimizer of \reff{4.11}, and $\phi_k = \phi^{\ast}$.
By the strong duality theorem, we have 
$\psi_k = \phi_k = \phi^{\ast}$ for all $k$ big enough.

The proofs for items (iii) and (iv) are the same. 
\end{proof}

%
%
In the following, we give some numerical examples.
We select the objectives as
\[
R_1=[x]_{d_1/2 }^TG_1^TG_1[x]_{ d_1/2 },~
R_2=[x]_{ d_1/2 }^TG_2^TG_2[x]_{ d_1/2  },
\]
where $G_1,\,G_2$ are randomly generated square matrices
whose entries obey the normal distribution.

\begin{exm}Consider the set
$$K = \{x\in\mathbb{R}^4: -x_1x_2x_3x_4\ge 0,\,-x_1-x_2\ge 0,-x_2-x_3\ge 0,\,\|x\|^2=1\}.$$
We look for a tensor $\mA \in\mt{S}^3(\mathbb{R}^4)$ satisfying equations
\[
\begin{gathered}
\a_{123}-\a_{444}=1,\quad \a_{234}-\a_{111}=1,\\
\a_{112}-\a_{333}=0,\quad \a_{221}-\a_{333}=0,\\
\a_{331}-\a_{444}=1,\quad \a_{334}-\a_{111}=1.
\end{gathered}
\]
At the relaxation order $k = 3$, we can recover that
$\a = \lambda_1(u_1)^{\otimes 3}-\lambda_2(u_2)^{\otimes 3}$, where
$$
\begin{array}{ll}
  u_1 = \begin{bmatrix}
     -0.7157&
    0.0002&
   -0.0806&
   -0.6938

\end{bmatrix}^T,   &  \lambda_1 =    2.7225,\\
 u_2=\begin{bmatrix}
     -0.0001&
   -0.1900&
   -0.2422&
    0.9511
\end{bmatrix}^T,    &  \lambda_2 =  0.1984.
\end{array}
$$
The computation took around 0.42s.
\end{exm}

\begin{exm}
Consider the set
\[
K=\{x\in\r^5:\,x_1+\cdots + x_5\ge 0,\,x_1x_2-x_3^2=0,\,x_3x_4-x_5^2=0,\,\|x\|^2=1\}.
\]
We look for a tensor $\a\in\mathbf{S}^3(\r^5)$ satisfying equations
\[
\begin{gathered}
\a_{223}-3\a_{111}+\a_{123}=2,\quad\a_{124}-\a_{333}+\a_{555}=9,\\
\a_{112}-3\a_{111}+\a_{123}=3,\quad \a_{344}-3\a_{111}+\a_{234}=2,\\
\a_{113}-3\a_{111}+\a_{123}=2,\quad \a_{155}-\a_{555}+\a_{235}=-12.
\end{gathered}
\]
For the relaxation order $k = 4$, we recover that
$\mA = \lambda_1(u_1)^{\otimes 3}+\lambda_2(u_2)^{\otimes 3}-\lambda_3(u_3)^{\otimes 3}-\lambda_4(u_4)^{\otimes 3}$, where
\[\begin{array}{ll}
    u_1 =\begin{bmatrix}
     0.5859&
    0.5679&
   -0.5768&
   -0.0025&
   -0.0382\end{bmatrix}^T, & \lambda_1 =   18.5321, \\
    u_2 = \begin{bmatrix}
    -0.4205&
   -0.4231&
    0.4218&
    0.5038&
    0.4610
\end{bmatrix}^T, & \lambda_2 =  43.6313.\\
 u_3 =\begin{bmatrix}
    0.1883&
    0.5077&
    0.3092&
    0.6423&
   -0.4457\end{bmatrix}^T, & \lambda_3 =10.5309, \\
 u_4 =\begin{bmatrix}
    0.9477&
    0.0920&
    0.2953&
    0.0195&
    0.0759\end{bmatrix}^T, & \lambda_4 =   1.9118. \\
\end{array}
\]
The computation took about 0.12s.
\end{exm}

\begin{exm} (random instances)
We consider the randomly generated tensor recovery. Let
$K = \{x\in\r^n:\,\|x\|^2= 1\}$ be the unit sphere.
We generate random $\f_1,\dots,\f_m$ whose entries obey the normal distribution. 
The obtained length $r$ is reported in Table~\ref{Table:03}.
One can see that  $r$ is typically much less than $m$.
\begin{table}[htb]
\centering
\caption{Values of the length $r$
for different $(n,d)$.}
\label{Table:03}
\begin{tabular}{|c|c|c|c|c|} \hline
$(n,d)$ & $m = 4$ & $m = 8$ & $m = 10$ & $m = 12$ \\  \hline
$(3,5)$  & $1, 2$   & $1,2,3,4,5$ & $2,3,4,6$ & $2,3,4,5,8$  \\ \hline
$(4,4)$  & $1, 2, 3 $  & $1,2,3, 4$ & $2,3,4$  & $2,3,4,5$  \\ \hline
$(6,6)$  & $1,2$   & $1,2,3$ & $1,2,4,5$  & $2,4,5,6$  \\ \hline
$(5,10)$  & $1,2$   & $1,2,3,4,5$ & $2,3,4,5,6$ & $2,3,4,5,6,7$  \\ \hline
\end{tabular}
\end{table}
\end{exm}

\section{Conclusion}

In this paper, we study moment and tensor recovery problems
whose decomposing vectors are contained in some constraining sets.
The Moment-SOS hierarchy is applied to 
solve the resulting moment optimization problems.
We choose generic SOS objective polynomials,
so that the obtained moment and tensor decompositions have low length.
The convergence properties of the Moment-SOS relaxations are studied.
Some numerical experiments are provided.

\bigskip 
\subsection*{Acknowledgements}
The authors are partially supported by the NSF grant DMS-2110780.


\begin{thebibliography}{99}

%
%

%
%


\bibitem{svd2}
\newblock J. Chen and Y. Saad,
\newblock \title{On the tensor SVD and the optimal low-rank orthogonal approximation of tensors},
\newblock \emph{SIAM Journal on Matrix Analysis and Applications}, \textbf{30}(2009), 1709-1734.


%
%

%
%


\bibitem{comon2008symmetric}
\newblock P.~Comon, G.~Golub, L.-H.~Lim and B.~Mourrain,
\newblock \title{Symmetric tensors and symmetric tensor rank},
\newblock \emph{SIAM Journal on Matrix Analysis and Applications},
\textbf{30}(2008), 1254-1279.



\bibitem{curto2005truncated}
\newblock R.E. Curto and L.A. Fialkow,
\newblock \title{Truncated K-moment problems in several variables},
\newblock \emph{Journal of Operator Theory}, (2005), 189-226.


\bibitem{castro1}
\newblock Y. De Castro and F. Gamboa,
\newblock \title{Exact reconstruction using Beurling minimal extrapolation},
\newblock \emph{Journal of Mathematical Analysis and applications},
\textbf{395}(2012), 336-354.


\bibitem{castro2}
\newblock Y. De Castro, F. Gamboa, D. Henrion and J.B. Lasserre,
\newblock \title{Exact solutions to super resolution on semi-algebraic domains in higher dimensions},
\newblock \emph{IEEE Transactions on Information Theory},
\textbf{63}(2016), 621-630.



\bibitem{dKL19}
\newblock E.~De~Klerk and M.~Laurent,
\newblock  \title{A survey of semidefinite programming approaches to the generalized problem of moments and their error analysis}
\newblock  In: {\em Araujo, C., Benkart, G., Praeger, C., Tanbay, B. (eds.) World Women in Mathematics 2018.}
Association for Women in Mathematics Series, vol. 20. Springer, Cham (2019).


\bibitem{FanZhou17}
\newblock J. Fan and A. Zhou,
\newblock \title{A semidefinite algorithm for completely positive tensor decomposition},
\newblock \emph{Computational Optimization and Applications} 66, 267--283, 2017.


\bibitem{FNZ19CP}
\newblock J. Fan, J. Nie and A. Zhou,
\newblock \title{Completely positive binary tensors},
\newblock \emph{Mathematics of Operations Research},
\textbf{44}(2019), 1087-1100.




\bibitem{truncated2}
\newblock L. Fialkow and J. Nie,
\newblock \title{The truncated moment problem via homogenization and flat extensions},
\newblock \emph{Journal of Functional Analysis}, \textbf{263}(2012), 1682-1700.

%
%



\bibitem{svdbased}
\newblock Y. Guan, M.T. Chu and D. Chu,
\newblock \title{SVD-based algorithms for the best rank-1 approximation of a symmetric tensor},
\newblock \emph{SIAM Journal on Matrix Analysis and Applications}, \textbf{39}(2018), 1095-1115.



%
%

\bibitem{gloptipoly}
\newblock D. Henrion, J.B. Lasserre and J. L\"{o}fberg,
\newblock \title{GloptiPoly 3: moments, optimization and semidefinite programming},
\newblock \emph{Optimization Methods and Software}, \textbf{24}(2009), 761-779.



\bibitem{HKL20}
\newblock D. Henrion,  M. Korda, and  J. Lasserre,
\newblock \emph{ The Moment-SOS Hierarchy: Lectures In Probability, Statistics,
Computational Geometry, Control And Nonlinear PDEs},
\newblock World Scientific, Cambridge, 2020.


%
%


\bibitem{hny1}
\newblock L. Huang, J. Nie and Y. Yuan,
\newblock \title{Homogenization for polynomial optimization with unbounded sets},
\newblock \emph{Mathematical Programming}, \textbf{200} (2023), 105-145.


\bibitem{hny2}
\newblock L. Huang, J. Nie and Y. Yuan,
\newblock \title{Finite convergence of Moment-SOS relaxations with non-real radical ideals},
\newblock preprint, 2023,


\bibitem{huang2023generalized}
\newblock L. Huang, J. Nie and Y. Yuan,
\newblock \title{Generalized truncated moment problems with unbounded sets},
\newblock \emph{Journal of Scientific Computing}, \textbf{95}(2023), 15.


%
%


%
%

\bibitem{existing1}
\newblock M. Korda, M. Laurent, V. Magron and A. Steenkamp,
\newblock \title{Exploiting ideal-sparsity in the generalized moment problem with application to matrix factorization ranks},
\newblock \emph{Mathematical Programming}, (2023), 1-42.



\bibitem{lasserre2008semidefinite}
\newblock J.B. Lasserre,
\newblock \title{A semidefinite programming approach to the generalized problem of moments},
\newblock \emph{Mathematical Programming}, \textbf{112}(2008), 65-92.


%
%


\bibitem{lasserre2015introduction}
\newblock J.B. Lasserre,
\newblock \emph{An introduction to polynomial and semi-algebraic optimization},
\newblock Cambridge University Press, 2015.


%
%


\bibitem{Lau05}
\newblock M.~Laurent,
\newblock \title{Revisiting two theorems of Curto and Fialkow on moment matrices},
\newblock \emph{Proceedings of the AMS},  133(10), 2965--2976, 2005.


\bibitem{sos}
\newblock M. Laurent,
\newblock \title{Sums of squares, moment matrices and optimization over polynomials},
\newblock \emph{Emerging Applications of Algebraic Geometry}, Springer, 2009, 157-270.

%
%



\bibitem{Lim13}
\newblock L.-H.~Lim,
\newblock \emph{Tensors and hypermatrices}, in: L. Hogben (Ed.),
Handbook of Linear Algebra, 2nd Ed.,
CRC Press, Boca Raton, FL, 2013.





\bibitem{successive}
\newblock C. Mu, D. Hsu, and D. Goldfarb,
\newblock \title{Successive rank-one approximations
for nearly orthogonally decomposable symmetric tensors},
\newblock \emph{SIAM Journal on Matrix Analysis and Applications},
\textbf{36}(2015), 1638-1659.



\bibitem{niecertificate}
\newblock J. Nie,
\newblock \title{Certifying convergence of Lasserre’s hierarchy via flat truncation},
\newblock \emph{Mathematical Programming}, \textbf{142}(2013), 485-510.


\bibitem{niecondition}
\newblock J. Nie,
\newblock \title{Optimality conditions and finite convergence of Lasserre’s hierarchy},
\newblock \emph{Mathematical Programming}, \textbf{146}(2014), 97-121.


\bibitem{nie2014truncated}
\newblock J. Nie,
\newblock \title{The $\mathcal{A}$-truncated $K$-moment problem},
\newblock \emph{Foundations of Computational Mathematics}, \textbf{14}(2014), 1243-1276.


\bibitem{nie2015linear}
\newblock J. Nie,
\newblock \title{Linear optimization with cones of moments and nonnegative polynomials},
\newblock \emph{Mathematical Programming}, \textbf{153}(2015), 247-274.


\bibitem{nietensor}
\newblock J. Nie,
\newblock \title{Symmetric tensor nuclear norms},
\newblock \emph{SIAM Journal on Applied Algebra and Geometry},
\textbf{1}(2017), 599-625.


\bibitem{GPSTD}
\newblock J. Nie,
\newblock \title{Generating polynomials and symmetric tensor decompositions},
\newblock \emph{Foundations of Computational Mathematics},
\textbf{17}(2), 423--465, 2017.



\bibitem{LRSTA17}
\newblock J. Nie,
\newblock \title{Low rank symmetric tensor approximations},
\newblock \emph{SIAM Journal on Matrix Analysis and Applications},
\textbf{38}(4), 1517--1540, 2017.




\bibitem{nie2023moment}
\newblock J. Nie,
\newblock \emph{Moment and Polynomial Optimization},
\newblock SIAM, 2023.


%
%

%


\bibitem{putinar1993positive}
\newblock M. Putinar,
\newblock \title{Positive polynomials on compact semi-algebraic sets},
\newblock \emph{Indiana University Mathematics Journal}, \textbf{42}(1993), 969-984.

%
%


\bibitem{reznick}
\newblock B. Reznick,
\newblock \title{Some concrete aspects of Hilbert's 17th problem},
\newblock \emph{Contemporary Mathematics} \textbf{253}(2000), 251-272.


%
%

%
%


\bibitem{sedumi}
\newblock J.F. Sturm,
\newblock \title{Using SeDuMi 1.02, a MATLAB toolbox for optimization over symmetric cones},
\newblock \emph{Optimization Methods and Software}, \textbf{11}(1999), 625-653.


%
%

%
%


\bibitem{svd3}
\newblock C.~Zeng and M.-K.~Ng,
\newblock \title{Decompositions of third‐order tensors: HOSVD, T‐SVD, and Beyond},
\newblock \emph{Numerical Linear Algebra with Applications}, \textbf{27}(2020), e2290.


%
%
%


\bibitem{ZhouFan14}
\newblock A.~Zhou and J.~Fan,
\newblock \title{The CP-matrix completion problem}, 
\newblock \emph{SIAM Journal on Matrix Analysis and Applications}, 
\textbf{35}(1), 127--142, 2014.


\bibitem{ZhouFan18}
\newblock A.~Zhou and J.~Fan,
\newblock \title{Completely positive tensor recovery with minimal nuclear value},  
\newblock \emph{Computational Optimization and Applications},
\textbf{70}, 419--441, 2018.



\end{thebibliography}
\end{document}